\documentclass[a4paper,12pt]{article}
\usepackage[centertags]{amsmath}
\usepackage{amsfonts}
\usepackage{amssymb}
\usepackage{amsthm}
\usepackage{amsmath}
\usepackage{dsfont}
\usepackage{graphicx}
\usepackage{tikz, subfigure}
\usepackage{pstricks-add}
\usepackage{verbatim}

\usepackage[latin1]{inputenc}
\usepackage{tikz}
\definecolor {processblue}{cmyk}{0.96,0,0,0}
\usepackage{amsfonts,graphicx,amsmath,amssymb,hyperref,color}
\usepackage[english]{babel}
\usepackage{authblk}

\newtheorem{Theorem}{Theorem}[section]
\newtheorem{Proposition}{Proposition}[section]
\newtheorem{Lemma}{Lemma}[section]
\newtheorem{Corollary}{Corollary}[section]

\newtheorem{Remark}{Remark}[section]

\AtEndDocument{\bigskip{\footnotesize%
  \textsc{Department of Mathematics, Ningbo University, Ningbo, Zhejiang, People's Republic of China} \par
  \textit{E-mail address:} \texttt{zhaobing@nbu.edu.cn} \par
}}
\AtEndDocument{\bigskip{\footnotesize%
  \textsc{Department of Mathematics, Ningbo University, Ningbo, Zhejiang, People's Republic of China} \par
  \textit{E-mail address:} \texttt{283186364@qq.com} \par
}}
\AtEndDocument{\bigskip{\footnotesize%
  \textsc{Department of Mathematics, Ningbo University, Ningbo, Zhejiang, People's Republic of China} \par
  \textit{E-mail address:} \texttt{740464942@qq.com} \par
}}
\AtEndDocument{\bigskip{\footnotesize%
  \textsc{Department of Mathematics, Ningbo University, Ningbo, Zhejiang, People's Republic of China} \par
  \textit{E-mail address:} \texttt{kanjiangbunnik@yahoo.com, jiangkan@nbu.edu.cn} \par
}}

\usepackage{lipsum}

\makeatletter
\newcommand*{\rom}[1]{\expandafter\@slowromancap\romannumeral #1@}
\makeatother
\begin{document}
\title{xxxx}
\date{}
 \title{Arithmetic on self-similar sets}
 \author{Bing Zhao, Xiaomin Ren, Jiali Zhu and Kan Jiang \thanks{Kan Jiang is the corresponding author}}
\maketitle{}
\begin{abstract}
Let $K_1$ and $K_2$ be two one-dimensional homogeneous self-similar sets. 
Let $f$
be a continuous function defined on an open set $U\subset \mathbb{R}^{2}$.
Denote the  continuous image of $f$  by
$$
f_{U}(K_1,K_2)=\{f(x,y):(x,y)\in (K_1\times K_2)\cap U\}.
$$
In this paper we give an sufficient condition which guarantees that $f_{U}(K_1,K_2)$ contains some interiors. Our result is different from Simon and Taylor's \cite[Proposition 2.9]{ST} as we do not need the condition that  the multiplication of the thickness of $K_1$ and $K_2$ is strictly greater than $1$. As a consequence, we give an application to the univoque sets in the setting of $q$-expansions.  

\end{abstract} 
\section{Introduction}
Arithmetic on some sets was  pioneered by Steinhuas  who proved the following classical result:  for any set $E\subset \mathbb{R}^n$ with positive Lebesgue measure, $E-E=\{x-y:x,y\in E\}$ contains interiors.    It is natural to consider the Steinhuas' result when $E$ is of zero Lebesgue measure.  Indeed, it is an important topic in fractal geometry and dynamical systems.  Let $C$ be   the middle-third Cantor set.
Let $C*C=\{x*y:x,y\in C\}$, where $*=+,-,\cdot,\div$ (when $*=\div$, we assume $y\neq 0$). 
 Steinhuas \cite{HS}  proved that $$C+C=[0,2],C-C=[-1,1].$$
Athreya,  Reznick, and Tyson \cite{Tyson} proved that $$C\div C=\cup_{-\infty}^{+\infty}3^{n}[2/3,3/2].$$  
The sum of  two Cantor sets appears naturally in  homoclinic bifurcations \cite{Palis}. Palis \cite{Palis} posed the following problem: whether it is true (at least generically) that the arithmetic  sum of dynamically defined Cantor sets either has measure zero or contains an interval. This conjecture was solved in \cite{Yoccoz}.  Motivated by Palis' conjecture, it is natural to investigate when the sum of two Cantor sets contains some interiors. Newhouse \cite{SN} proved the following celebrated results. Given any two Cantor sets $C_1$ and $C_2$,   if $\tau(C_1)\tau(C_2)>1$, where $\tau(C_i),i=1,2$ denotes the thickness of $C_i, i=1,2,$ then $C_1+C_2$ contains some interiors.  Motivated by Newhouse's result,  Simon and Taylor\cite[Proposition 2.9]{ST}  proved  that for any $C^2$ functions 
$H(x,y)$ defined on some open set $U$, if two Cantor sets $C_1$  and  $C_2$  have  the property  $\tau(C_1)\tau(C_2)>1$,  and the partial derivatives of $H(x,y)$ are not vanishing  for 
almost everywhere in $U$, then $H(C_1, C_2)=\{H(x,y):x\in C_1, y\in  C_2\}$ contains some intervals. 
Simon and Taylor's result is elegant, however,  their result is not true   when $\tau(C_1)\tau(C_2)\leq 1$.  The theme of this paper is to weaken this condition for some fractal sets.  

 Let $f$
be a continuous function defined on an open set $U\subset \mathbb{R}^{2}$, and $E,F$ be two non-empty sets in $\mathbb{R}$.
Denote the  continuous image of $f$ by
$$
f_{U}(E,F)=\{f(x,y):(x,y)\in (E\times F)\cap U\}.
$$
 Let  $K_1$ and $K_2$ be the attractors of the   IFS's $\{f_i(x)=\lambda x+a_i\}_{i=1}^{n}$ and $\{g_j(x)=\lambda x+b_j\}_{j=1}^{m}$, respectively, where $ a_i,b_j\in \mathbb{R}, 0<\lambda<1.$  
   Let the convex hull of $K_1$($K_2$) be $[a,b]$($[c,d]$). Without loss of generality, we may assume $ f_1(a)=a,  f_n(b)=b$ ($ g_1(c)=c,  g_m(d)=d$), and  $$f_i(a)\leq  f_{i+1}(a),1\leq  i\leq n-1, g_j(c)\leq  g_{j+1}(c),1\leq  j\leq m-1.$$
     Let $$D_1=\{1\leq i\leq n-1:f_i(b)-f_{i+1}(a)<0\}, \kappa_1=\max_{i\in D_1}\{f_{i+1}(a)-f_i(b)\}.$$
     $$D_2=\{1\leq j\leq m-1:g_j(d)-g_{j+1}(c)<0\}, \kappa_2=\max_{j\in D_2}\{g_{j+1}(c)-g_j(d)\}.$$
\begin{Theorem}\label{Main2}
Let $K_1,K_2$ be the attractors  defined as above.  
If there is some $(x_0,y_0)\in (K_1\times K_2)\cap U$ such that  
$$\dfrac{\kappa_1}{d-c}<\left\vert \dfrac{\partial_y f|_{(x_0,y_0)}}{\partial_x f|_{(x_0,y_0)}}\right\vert <\dfrac{\lambda(b-a)}{\kappa_2},$$
then $f_U(K_1,K_2)$ contains some interiors. 
\end{Theorem}
The following result can be obtained easily in terms of Theorem \ref{Main2}. 
\setcounter{Corollary}{1}
\begin{Corollary}
Let $K_1,K_2$ be the attractors  defined as above.  If \begin{equation*}
\left\lbrace\begin{array}{cc}
               \lambda(b-a)>\kappa_2\\
                \kappa_1<d-c,\\
                \end{array}\right.
\end{equation*}  then for $f(x,y)=x^\alpha\pm y^{\alpha}, \alpha\neq 0, xy, \dfrac{x}{y},  $
$f_U(K_1,K_2)$ contains some interiors. 
\end{Corollary}
\setcounter{Remark}{2}
\begin{Remark}
We have the following remarks.
\begin{itemize}
\item [(1)]  We compare Theorem \ref{Main2} with Simon and Taylor's result.
First, our result can be  checked directedly.    More importantly, we do not need the condition that   the multiplication of the  thickness is strictly greater than $1$.  

\item [(2)] For the  irreducible graph-directed self-similar set (inhomogeneous self-similar set), we are able to construct a homogeneous self-similar set such that it can  arbitrarily approximate the irreducible graph-directed self-similar set (inhomogeneous self-similar set) in the sense of Hausdorff dimension \cite[Theorem 1.1]{Farkas}, \cite[Proposition 6]{PS}. Thus, we can consider similar problem on the 
 irreducible graph-directed self-similar set (inhomogeneous self-similar set).
 
\item [(3)]  Theorem  \ref{Main2} is very useful to many other fractal sets, for instance, the set of points with multiple codings,  some attractors generated by the open dynamics  and so forth. 
The usage of Theorem \ref{Main2} is as follows: for a given fractal set, we firstly find a sub self-similar set of this fractal set (if it is an inhomogeneous self-similar set, then we use the second remark above), then we utilize Theorem \ref{Main2} to obtain partial results.  We emphasize here that  \cite[Theorem 1.1]{Farkas}, \cite[Proposition 6]{PS} is very helpful to analyze the general fractal sets. 
\end{itemize}
\end{Remark}
We give an application to $q$-expansions. 
Let $1<q<2$. It is well-known that  for every $x\in\left[0,\dfrac{1}{q-1}\right]$  there is some $(a_n)\in\{0,1\}^{\mathbb{N}}$  such that 
$$x=\sum_{n=1}^{\infty}\dfrac{a_n}{q^n}.
$$
We call $(a_n)$ an $q$-expansion of  $x$. Typically, $x$ has multiple expansions \cite{Sidorov,KarmaMartijn}. If $x$ has a unique expansion, then we call $x$ a univoque point. 
Write $U_q$ for the set of points with unique expansions.  Sidorov proved in \cite{SN} that for any $q\geq T_3$, then 
$$U_q+U_q=\left[0,\dfrac{2}{q-1}\right],$$
where $T_3$ is the Pisot number satisfying 
$x^3=x^2+x+1.$ He also posed a question  that finding the smallest base $q$ such that 
$U_q+U_q=\left[0,\dfrac{2}{q-1}\right]$.  In \cite{DKKL}, Dajani, Komornik, Kong and Li proved   that $\mathcal{U}(x)*\mathcal{U}(x)$ contains interiors, where $*=+ ,-,\cdot , \div$, and $\mathcal{U}(x)$ is the set of bases for which the expansion of $x\in(0,1]$ is unique.  The main tool of proving  the above two results is the Newhouse thickness theorem. Motivated by their results, it is natural to consider the arithmetic on the univoque set $U_q$.  To the best of our knowledge,  there are very few results concerning with   $f_{U}(U_q,U_q)$, where $f(x,y)$ is a general function. 
The following results are corollaries of Theorem \ref{Main2}. 
\setcounter{Corollary}{3}
\begin{Corollary}\label{CCC}
Let  $q_{*}\approx 1.8019$ be the appropriate root of 
$$x^3-x^2-2x+1=0.$$ For any $q\in(q_{*},2)$, 
if there is some $(x_0,y_0)\in (K_q\times K_q)\cap U$ such that 
\begin{equation*}
1-2\lambda<\left\vert \frac{\partial
_{y}f|_{(x_{0},y_{0})}}{\partial _{x}f|_{(x_{0},y_{0})}}\right\vert <\dfrac{\lambda}{1-2\lambda},
\end{equation*}
where $\lambda=q^{-2}$ and $K_q$ is the attractor of the IFS 
$$\left\{g_1(x)=\dfrac{x+1}{q^2}, g_2(x)=\dfrac{x}{q^2}+\dfrac{1}{q} \right\},$$ then  
 $f_{U}(U_q,U_q)$ contains an interior. 
\end{Corollary}
\begin{Corollary}\label{Cor1}
If  $q\in(q_{*},2)$, then $$ U_q^2+ U_q^2, U_q^2- U_q^2, U_q\cdot U_q, U_q\div U_q$$ have interiors. 
\end{Corollary}
It is natural to consider the dimension of $f_U(U_q, U_q)$. The following results are indeed some  corollaries  of B{\'a}r{\'a}ny \cite{BABA1}, Peres and Shmerkin\cite{PS}.
\setcounter{Proposition}{5}
\begin{Proposition}\label{pro}
For any $1<q<2$, 
$$\dim_{H}(U_q\cdot U_q)=\dim_{H}(U_q\div U_q)=\dim_{H}(U_q^2\pm U_q^2)=\min\{2\dim_{H}(U_q),1\}.$$
For two different bases  $q_1,q_2\in(1,2)$ with the property $$\dfrac{\log q_1}{\log q_2}\notin \mathbb{Q},$$ then 
$$\dim_{H}(U_{q_1}+ U_{q_2})=\min\{\dim_{H}(U_{q_1})+\dim_{H}(U_{q_2}), 1\}.$$ 
\end{Proposition}
The paper is arranged as follows. In section 2, we give the proofs of main results. In section 3, we give some remarks. 

\section{Proof of Main results}
In this section, we shall give the proofs of the main theorems of this paper. 
\subsection{Proof of Theorem \ref{Main2}}
Before, we prove Theorem \ref{Main2}, we recall some definitions. For simplicity, we only  introduce the definitions for $K_1$.  For any $(i_{1},\cdots ,i_{k})\in \{1,2,\cdots, n\}^{k}$, we call $
f_{i_{1},\cdots ,i_{k}}([a,b])=(f_{i_{1}}\circ \cdots \circ f_{i_{k}})([a,b])$ a
basic interval of rank $k,$ which has length $\lambda^{k}(b-a)$.  In  what follows, we use the following notation
$f_{i_{1}}\circ \cdots \circ f_{i_{k}}=f_{i_{1},\cdots ,i_{k}}.$
Denote by $H_{k}$ the collection of all these basic intervals
of rank $k$ with respect to $\{f_i\}_{i=1}^{n}$. Let $J\in H_{k}$, define $\widetilde{J}=\cup _{i=1}^{n}I_{k+1,i}$
, where $I_{k+1,i}\in H_{k+1}$ and $I_{k+1,i}\subset J$ for $i=1,2,\cdots n$. 
Let $%
[A,B]\subset \lbrack a,b]$, where $A$ and $B$ are the left and right
endpoints of some basic intervals in $H_{k}$ for some $k\geq 1$,
respectively. $A$ and $B$ may not be in the same basic interval. Let $F_{k}$
be the collection of all the basic intervals in $[A,B]$ with length $%
\lambda^{k}(b-a),k\geq k_{0}$ for some $k_{0}\in \mathbb{N}^{+}$, i.e. the union of
all the elements of $F_{k}$ is denoted by $G_{k}=\cup _{i=1}^{t_{k}}I_{k,i}$%
, where $t_{k}\in \mathbb{N}^{+}$, $I_{k,i}\in H_{k}$ and $I_{k,i}\subset
\lbrack A,B]$. Clearly, by the definition of $G_{k}$, it follows that $%
G_{k+1}\subset G_{k}$ for any $k\geq k_{0}.$ Similarly, we can define the basic intervals for $K_2$. Let  $M$ and
$N$ are the left and right endpoints of some basic intervals in $H_{k}^{'}$ with respect to $K_2$.
Denote by $G_{k}^{\prime }$ the union of all the basic intervals with length
$\lambda^{k}(d-c)$ in the interval $[M,N]$, i.e. $G_{k}^{\prime }=\cup
_{j=1}^{t_{k}^{\prime }}J_{k,j}^{\prime }$, where $t_{k}^{\prime }\in \mathbb{N}^{+}$,
$J_{k,j}^{\prime }\in H_{k}^{\prime }$ and $J_{k,j}^{\prime }\subset \lbrack M,N]$.
The following lemma is crucial to our analysis. 

\begin{Lemma}\label{key1} 
Let $f:U\rightarrow \mathbb{R}$ be a continuous function$.$
Suppose $A$ and $B$ ($M$ and $N$) are the left and right endpoints of some
basic intervals in $H_{k_{0}}$ ($H_{k_{0}}^{\prime }$) for some $k_{0}\geq 1$ respectively such that
$[A,B]\times \lbrack M,N]\subset U.$ Then $K_1\cap \lbrack A,B]=\cap _{k={k_{0}%
}}^{\infty }G_{k}$, and $K_2\cap \lbrack M,N]=\cap _{k={k_{0}}}^{\infty
}G_{k}^{\prime }$. Moreover, if for any $k\geq k_{0}$ and any two basic
intervals $I\subset G_{k}$, $J\subset G_{k}^{\prime }$ such that
\begin{equation*}
f(I,J)=f(\widetilde{I},\widetilde{J}),
\end{equation*}%
then $f(K_1\cap \lbrack A,B],K_2\cap \lbrack M,N])=f(G_{k_{0}},G_{k_{0}}^{\prime
}).$
\end{Lemma}
\begin{proof}
By the construction of $G_{k}$ ($G_{k}^{\prime }$), i.e. $G_{k+1}\subset
G_{k}$ ($G_{k+1}^{\prime }\subset G_{k}^{\prime }$) for any $k\geq k_{0}$,
it follows that
\begin{equation*}
K_1\cap \lbrack A,B]=\cap _{k=k_{0}}^{\infty }G_{k}\text{ and }K_2\cap \lbrack
M,N]=\cap _{k=k_{0}}^{\infty }G_{k}^{\prime }.
\end{equation*}%
The continuity of $f$ yields that
\begin{equation*}
f(K_1\cap \lbrack A,B],K_2\cap \lbrack M,N])=\cap _{k=k_{0}}^{\infty
}f(G_{k},G_{k}^{\prime }).
\end{equation*}%
In terms of the relation $G_{k+1}=\widetilde{G_{k}}$, $G_{k+1}^{\prime }=%
\widetilde{G_{k}^{\prime }}$ and the condition in the lemma, it follows that
\begin{eqnarray*}
f(G_{k},G_{k}^{\prime }) &=&\cup _{1\leq i\leq t_{k}}\cup _{1\leq j\leq
t_{k}^{\prime }}f(I_{k,i},J_{k,j}) \\
&=&\cup _{1\leq i\leq t_{k}}\cup _{1\leq j\leq t_{k}^{\prime }}f(\widetilde{%
I_{k,i}},\widetilde{J_{k,j}}) \\
&=&f(\cup _{1\leq i\leq t_{k}}\widetilde{I_{k,i}},\cup _{1\leq j\leq
t_{k}^{\prime }}\widetilde{J_{k,j}}) \\
&=&f(G_{k+1},G_{k+1}^{\prime }).
\end{eqnarray*}%
Therefore, $f(K_1\cap \lbrack A,B],K_2\cap \lbrack
M,N])=f(G_{k_{0}},G_{k_{0}}^{\prime }).$
\end{proof}
\setcounter{Theorem}{1}
\begin{Theorem}\label{thm1}
Suppose that  $\partial _{x}f, \partial _{y}f$ are   continuous.
If there is some $(x_0,y_0)\in (K_1\times K_2)\cap U$ such that $$\dfrac{\kappa_1}{d-c}<\left\vert \dfrac{\partial_y f|_{(x_0,y_0)}}{\partial_x f|_{(x_0,y_0)}}\right\vert <\dfrac{\lambda(b-a)}{\kappa_2},\partial_y f|_{(x_0,y_0)}>0, \partial_x f|_{(x_0,y_0)}>0,$$
then $f_U(K_1,K_2)$ contains some interiors. 
\end{Theorem}
\begin{proof}
Since $\partial _{x}f, \partial _{y}f$ are   continuous, and $\partial_y f|_{(x_0,y_0)}>0, \partial_x f|_{(x_0,y_0)}>0$, it follows that there is some neighbourhood of $(x_0,y_0)$, denoted by $V$,  such that 
$$\partial_y f(x,y)>0, \partial_x f(x,y)>0$$
for any $(x,y)\in V.$
Let $I=f_{i_1\cdots i_p}([a,b])$ and $J=g_{j_1\cdots j_p}([c,d])$ be two basic intervals of $H_k$ and $H_k^{\prime}$, respectively.  We assume that 
$I\times J\subset V$. 
By the definition $\widetilde{I}$ and $\widetilde{J}$, we have 
$$\widetilde{I}=\cup_{j=1}^{n}f_{i_1\cdots i_pj}([a,b]), \widetilde{J}=\cup_{l=1}^{m}g_{j_1\cdots j_pl}([c,d]).$$
Therefore, $$f(\widetilde{I},\widetilde{J})=\cup_{j=1}^{n}\cup_{l=1}^{m}f(f_{i_1\cdots i_pj}([a,b]),g_{j_1\cdots j_pl}([c,d])).$$
For any $1\leq j\leq n$, we shall prove that 
$$\cup_{l=1}^{m}f(f_{i_1\cdots i_pj}([a,b]),g_{j_1\cdots j_pl}([c,d]))$$ is an interval. In fact, by the conditions $\partial_y f|_{(x_0,y_0)}>0, \partial_x f|_{(x_0,y_0)}>0$,   it remains to prove that 
$$f(f_{i_1\cdots i_pj}(b),g_{j_1\cdots j_pl}(d))-f(f_{i_1\cdots i_pj}(a),g_{j_1\cdots j_pl+1}(c))\geq 0.$$
Note that 
\begin{eqnarray*}
&& f(f_{i_1\cdots i_pj}(b),g_{j_1\cdots j_pl}(d))-f(f_{i_1\cdots i_pj}(a),g_{j_1\cdots j_pl+1}(c))\\
&=&(f_j(b)-f_j(a)) \lambda^{p}\partial_x f+ (g_l(d)-g_{l+1}(c))\lambda^{p}\partial_y f +o(\lambda^{p})] \\
&=&\lambda^{p}\left[ (b-a)\lambda  \partial_x f(x,y)+(g_l(d)-g_{l+1}(c))\partial_y f(x,y) +o(1)\right]
\end{eqnarray*}
for any $(x,y)\in I\times J\subset V$. 
By the condition $\dfrac{\partial_y f|_{(x_0,y_0)}}{\partial_x f|_{(x_0,y_0)}} <\dfrac{\lambda(b-a)}{\kappa_2}$  and the continuity  of $\partial_x f, \partial_y f$, it follows that $$f(f_{i_1\cdots i_pj}(b),g_{j_1\cdots j_pl}(d))-f(f_{i_1\cdots i_pj}(a),g_{j_1\cdots j_pl+1}(c))\geq 0.$$
Next, we want to show that 
$$\cup_{l=1}^{m}f(f_{i_1\cdots i_pj}([a,b]),g_{j_1\cdots j_pl}([c,d]))\cup 
\cup_{l=1}^{m}f(f_{i_1\cdots i_pj+1}([a,b]),g_{j_1\cdots j_pl}([c,d]))$$ is an interval for any $1\leq j\leq n-1$. 
Indeed, it suffices to prove that 
\begin{eqnarray*}
&& f(f_{i_1\cdots i_pj}(b),g_{j_1\cdots j_pm}(d))-f(f_{i_1\cdots i_pj+1}(a),g_{j_1\cdots j_p1}(c))\\
&=&(f_j(b)-f_{j+1}(a)) \lambda^{p}\partial_x f+ (g_m(d)-g_{1}(c))\lambda^{p}\partial_y f +o(\lambda^{p})] \\
&=&\lambda^{p}\left[ (f_j(b)-f_{j+1}(a))\partial_x f+(d-c)\partial_y f +o(1)\right]\\
&=&\lambda^{p}\left[ (f_j(b)-f_{j+1}(a))\partial_x f+(d-c)\partial_y f +o(1)\right]
\end{eqnarray*}
Clearly if $f_j(b)-f_{j+1}(a)\geq 0$, then 
$$\lambda^{p}\left[ (f_j(b)-f_{j+1}(a))\partial_x f+(d-c)\partial_y f +o(1)\right]\geq 0$$
 If $f_j(b)-f_{j+1}(a)<0$, i.e. $j\in D_1$, then by the condition $\dfrac{\kappa_1}{b-a}<\dfrac{\partial_y f|_{(x_0,y_0)}}{\partial_x f|_{(x_0,y_0)}}$  and the continuity  of $\partial_x f, \partial_y f$,  we also have 
 $$\lambda^{p}\left[ (f_j(b)-f_{j+1}(a))\partial_x f+(d-c)\partial_y f +o(1)\right]\geq 0$$
Therefore, we  have proved that $f(I,J)=f(\widetilde{I}, \widetilde{J})$. In terms of Lemma \ref{key1}, it follows that $f_U(K_1,K_2)$ contains some interiors.
\end{proof}
Similarly, we can prove the following result. We left it to the readers. 
\begin{Theorem}\label{thm2}
Suppose that  $\partial _{x}f, \partial _{y}f$ are   continuous.
If there is some $(x_0,y_0)\in (K_1\times K_2)\cap U$ such that $$\dfrac{\kappa_1}{d-c}<\left\vert \dfrac{\partial_y f|_{(x_0,y_0)}}{\partial_x f|_{(x_0,y_0)}}\right\vert <\dfrac{\lambda(b-a)}{\kappa_2},\partial_y f|_{(x_0,y_0)}>0, \partial_x f|_{(x_0,y_0)}<0,$$
then $f_U(K_1,K_2)$ contains some interiors. 
\end{Theorem}
\begin{proof}[\textbf{Proof of Theorem \ref{Main2}}]
By Theorems \ref{thm1} and \ref{thm2}, it remains to prove the following two cases.
\begin{itemize}
\item [(1)] If $\partial_y f|_{(x_0,y_0)}<0, \partial_x f|_{(x_0,y_0)}<0,$ then we let $F_1(x,y)=-f(x,y)$ and use Theorem \ref{thm1}. 
\item [(2)] If $\partial_y f|_{(x_0,y_0)}>0, \partial_x f|_{(x_0,y_0)}<0,$ then we let $F_2(x,y)=-f(x,y)$, and make use of Theorem \ref{thm2}. 
\end{itemize}
\end{proof}
\subsection{\textbf{Proof of Theorem \ref{CCC}}}
Now, we prove Theorem \ref{CCC}. Our idea is simple, i.e. for any $q\in (q_{*}, 2)$, we shall construct some self-similar set, denoted by $K_q$, contained in $U_q$ such that $f_U(K_q,K_q)$ contains some interiors. Therefore, $f_U(U_q,U_q)$ has some interiors. 
Before we construct the set $K_q$ we give some classical results of unique expansions. 
The following theorem characterizes the criteria of the unique expansions, the proof of this result can be found in \cite{MK} or  some references therein. 
\begin{Theorem}\label{Uniquecodings}
 Let $(a_n)_{n=1}^{\infty}$ be an expansion of $x\in [0, (q-1)^{-1}]$. Then $(a_n)_{n=1}^{\infty}\in \widetilde{U}_{q}$ if and only if 
 $$(a_{k+1}a_{k+2}\cdots)<(\eta_n)_{n=1}^{\infty}$$ wherever $a_k=0$, 
  $$\overline{(a_{k+1}a_{k+2}\cdots)}<(\eta_n)_{n=1}^{\infty}$$ wherever $a_k=1$, 
  where $(\eta_n)_{n=1}^{\infty}$ is the quasi-greedy expansions of $1$,   $\widetilde{U}_{q}$ denotes all the unique expansions in base $q$, and $``<"$ means the lexicographic order. 
\end{Theorem}
\setcounter{Lemma}{4}
\begin{Lemma}
Let  $q_{*}\approx 1.8019$ be the appropriate root of 
$$x^3-x^2-2x+1=0.$$  Then   for any  $q\in (q_{*}, 2)$,  $$U_{q}\supset K_q,$$ where 
$K_q$ is the attractor  of the following IFS
$$\left\{g_1(x)=\dfrac{x+1}{q^2}, g_2(x)=\dfrac{x}{q^2}+\dfrac{1}{q} \right\}.$$
\end{Lemma}
\begin{proof}
In base $q_{*}$, the quasi-greedy expansion of $1$ is $11(01)^{\infty}.$ Therefore, by Theorem \ref{Uniquecodings}, we have that   for any $q\in (q_{*}, 2)$,   $$\{(01), (10)\}^{\infty}\subset \widetilde{U_{q_{*}}}\subset \widetilde{U_{q}}.$$
Thus, we can construct a self-similar set $K_q$ by the coding space $\{(01), (10)\}^{\infty}$, namely, 
$$U_{q}\supset K_q,$$
where $K_q$ is the attractor of the IFS
$$\left\{g_1(x)=\dfrac{x+1}{q^2}, g_2(x)=\dfrac{x}{q^2}+\dfrac{1}{q} \right\}.$$
\end{proof}
Now, we are able to prove Corollary \ref{CCC}.
\begin{proof}[\textbf{Proof of Corollary \ref{CCC}}]
For the self-similar set $K_q$, note that $a=\dfrac{1}{q^2-1},b=\dfrac{q}{q^2-1}$, 
$r=q^{-2}$. Therefore, the condition $$\dfrac{\kappa_1}{b-a}<\left\vert \dfrac{\partial_y f|_{(x_0,y_0)}}{\partial_x f|_{(x_0,y_0)}}\right\vert <\dfrac{\lambda(b-a)}{\kappa_1}$$ is exactly the following condition
\begin{equation*}
1-2\lambda<\left\vert \frac{\partial
_{y}f|_{(x_{0},y_{0})}}{\partial _{x}f|_{(x_{0},y_{0})}}\right\vert <\dfrac{\lambda}{1-2\lambda},
\end{equation*}
where $\lambda=q^{-2}$.
Therefore, we prove Corollary \ref{CCC}.
\end{proof}
\begin{proof}[\textbf{Proof of Corollary \ref{Cor1}}]
For $f(x,y)= xy \mbox{ or }x^2+y^2 \mbox{ or } x^2-y^2, y/x$
$$ \left\vert\frac{\partial
_{x}f|_{(x_{0},y_{0})}}{\partial _{y}f|_{(x_{0},y_{0})}}\right \vert=\dfrac{y_0}{x_0}\mbox{ or } \left\vert\frac{\partial
_{x}f|_{(x_{0},y_{0})}}{\partial _{y}f|_{(x_{0},y_{0})}}\right\vert=\dfrac{x_0}{y_0}.$$
Without loss of generality, we only consider the following case as for the remaining case, the discussion is analogous. Suppose  
$$ \left\vert\frac{\partial
_{x}f|_{(x_{0},y_{0})}}{\partial _{y}f|_{(x_{0},y_{0})}}\right\vert=\dfrac{y_0}{x_0}.$$
We take $(x_0, y_0)=\left(\dfrac{1}{q^2-1},\dfrac{q}{q^2-1}\right)\in K_q\times K_q$. It is easy to check that in this case 
$$ \dfrac{1-2\lambda}{\lambda}<\left\vert\frac{\partial
_{x}f|_{(x_{1},y_{1})}}{\partial _{y}f|_{(x_{1},y_{1})}}\right\vert=\dfrac{y_1}{x_1}<\dfrac{1}{1-2\lambda},$$
for any $q\in(q_*,2)$, where $\lambda=q^{-2}.$
Therefore, Corollary \ref{Cor1} follows Corollary \ref{CCC}. 
\end{proof}
\subsection{\textbf{Proof of Proposition \ref{pro}}}
B{\'a}r{\'a}ny  \cite{BABA1} proved the following result.
\setcounter{Theorem}{5}
\begin{Theorem}\label{BBB}
Let $\Lambda$ be an arbitrary self-similar set in $\mathbb{R}^2$ not contain in any line. Suppose that 
$g:\mathbb{R}^2\to \mathbb{R}$ is a $C^2$ map such that 
$$(g_x)^2+(g_y)^2\neq 0, (g_{xx}g_y-g_{xy}g_x)^2+(g_{xy}g_y-g_{yy}g_x)^2\neq 0$$
for any $(x,y)\in \Lambda$. Then 
$$\dim_{H}g(\Lambda)=\min\{1,\dim_{H}(\Lambda)\}.$$
\end{Theorem}
In terms of this result, we can prove Proposition \ref{pro}.
\begin{proof}[\textbf{Proof of  Proposition \ref{pro}}]
First, we prove that 
for any $1<q<2$, 
$$\dim_{H}(U_q\cdot U_q)=\dim_{H}(U_q\div U_q)=\dim_{H}(U_q^2\pm U_q^2)=\min\{2\dim_{H}(U_q),1\}.$$
It remains to prove that the formula  is correct for any $q_{KL}<q<2$, where $q_{KL}$ is the Komornik-Loreti  constant. As for any $q\in(1,q_{KL}]$, $\dim_{H}(U_q)=0$ (see  \cite{GS}),
and the fact that for any   $C^2$ map $g$, we always have  
$$\dim_{H}(g(U_{q},  U_{q}))\leq \min\{2\dim_{H}(U_{q}),1\}.$$
In our proposition, $g(x,y)=xy, x^2\pm y^2, x/y$.  In what follows, we always assume $g(x,y)$ is one of the four functions. 
For $f(x,y)=x/y$, we may adjust the proof if  necessary. 
Let $\widetilde{U_q}$ be the associated coding space of $U_q$. For any $\epsilon>0$, we may find a subshift of finite type, denoted by $\widetilde{U_{q_0}},$  with transitivity  condition \cite{RSK} such that  $\widetilde{U_{q_0}}\subset \widetilde{U_q}$ and that 
$$ \dim_{H}(\pi_{q}(\widetilde{U_{q_0}}))>\dim_{H}(U_{q})-\epsilon,$$
where $\pi_{q}(\cdot)$ means the natural projection of  the coding space in base $q$.
Note that $\pi_{q}(\widetilde{U_{q_0}})$ is a graph-directed self-similar set. We denote it by $(K_1,K_2,\cdots, K_n)$. By the transitivity  condition, it follows that 
$$\dim_{H}(\pi_{q}(\widetilde{U_{q_0}}))=\dim_{H}(K_i),$$for any $1\leq i\leq n.$
Given $K_i$,  by Theorem \cite[Theorem 1.1]{Farkas}, for any $\epsilon>0$, there is some self-similar set $K$ such that 
$$K\subset K_i, \dim_{H}(K)>\dim_{H}(K_i)-\epsilon.$$ 
Therefore, 
for any $q_{KL}<q<2$, and any $\epsilon>0$ there exists some homogeneous self-similar set $K\subset U_q$ such that 
$$\dim_{H}(K)>\dim_{H}(U_q)-\epsilon.$$
Note that $K\times K$ is a self-similar set in $\mathbb{R}^2$ which is not contained in a line. By Theorem \ref{BBB}
$$\dim_{H}g(K\times K)=\min\{1,\dim_{H}(K\times K)\}.$$
Here if $g(x,y)=x/y$, we may replace $U_q$ by $\left[\dfrac{2-q}{q-1},1\right]\cap U_q$. As for this case, $0\in U_q$ and we need to avoid this point if we use Theorem \ref{BBB}.  The rest proof is the same. 
Therefore, 
 \begin{eqnarray*}
\dim_{H}(g(U_{q},  U_{q}))&\geq& \dim_{H}(g(K,K))\\&=
&\min\{\dim_{H}(K\times K),1\}\\
&=&\min\{2\dim_{H}(K),1\}\\
 &\geq& \min\{2(\dim_{H}(U_{q})-\epsilon),1\}.
\end{eqnarray*}
Letting $\epsilon\to 0$, and we obtain that 
$$\dim_{H}(g(U_{q},  U_{q}))\geq \min\{2\dim_{H}(U_{q}),1\},$$
and subsequently we have proved that $$\dim_{H}(g(U_{q},  U_{q}))= \min\{2\dim_{H}(U_{q}),1\}.$$
The proof of the remaining formula is similar. 
With a similar discussion, for any $q_1, q_2\in (q_{KL}, 2)$
we have
$$
\dim_{H}(U_{q_1}+ U_{q_2})=\dim _HP(U_{q_1}\times U_{q_2})\leq \dim _H(U_{q_1}\times U_{q_2})=\dim_{H}(U_{q_1})+\dim_{H}(U_{q_2}),
$$
where $P$ denotes the projection to the $y$-axis through the angle $3\pi/4.$
Thus  
$$\dim_{H}((U_{q_1}+ U_{q_2}))\leq \min\{\dim_{H}(U_{q_1})+\dim_{H}(U_{q_2}), 1\}.$$
In order to show that $\dim_{H}((U_{q_1}+ U_{q_2}))\geq \min\{\dim_{H}(U_{q_1})+\dim_{H}(U_{q_2}), 1\}$.
Arbitrarily   fix an $\epsilon >0$. As in the proof of 
$$\dim_{H}(g(U_{q},  U_{q}))\geq \min\{2\dim_{H}(U_{q}),1\},$$
one can take   graph-directed self-similar sets $\pi_{q_1}(\widetilde{U_1})$ and $\pi_{q_2}(\widetilde{U_2})$ such that 
$\pi_{q_1}(\widetilde{U_1})\subseteq U_{q_1}$, $\pi_{q_2}(\widetilde{U_2})\subseteq U_{q_2}$ and 
$$ 
\dim_{H}(\pi_{q_1}(\widetilde{U_1}))>\dim_{H}(U_{q_1})-\epsilon, \;\; \dim_{H}(\pi_{q_2}(\widetilde{U_2}))>\dim_{H}(U_{q_2})-\epsilon .
$$
By \cite[Remark 1.2]{Farkas}, there exist some homogeneous self-similar sets $K_{q_1}$ and $K_{q_2}$ with similarity ratios $1/q_1$ and $1/q_2$, respectively, such that 
$$K_{q_1}\subset \pi_{q_1}(\widetilde{U_1}), K_{q_2}\subset \pi_{q_2}(\widetilde{U_2})$$
and that 
$$\dim_{H}(K_{q_1})>\dim_{H}(\pi_{q_1}(\widetilde{U_1}))-\epsilon, \dim_{H}(K_{q_2})>\dim_{H}(\pi_{q_2}(\widetilde{U_2}))-\epsilon.$$
 By \cite[Theorem 2]{PS}, we have 
 $$\dim_{H}(K_{q_1}+K_{q_2})=\min\{\dim_{H}(K_{q_1})+\dim_{H}(K_{q_2}),1\}$$
Therefore, 
$$
\dim_{H}(U_{q_1}+ U_{q_2})\geq \dim_{H}(\pi_{q_1}(\widetilde{U_1})+ \pi_{q_2}(\widetilde{U_2}))
\geq \min\{\dim_{H}(U_{q_1})+\dim_{H}(U_{q_2})-4\epsilon , 1\}.
$$
Therefore, $$\dim_{H}((U_{q_1}+ U_{q_2}))=\min\{\dim_{H}(U_{q_1})+\dim_{H}(U_{q_2}), 1\}.$$
\end{proof}
\section{Final remarks}
It would be interesting if one can give the exact form of  $f_U(K_1, K_2)$ for some functions. For instance, what is the exact structure of $C\cdot C$? We shall give an answer to this question in another paper. 

 \section*{Acknowledgements}
The work is supported by National Natural Science Foundation of China (Nos.11701302,

11671147). The work is also supported by K.C. Wong Magna Fund in Ningbo University. Kan Jiang would like to thank Karma Dajani, Derong Kong and Wenxia Li for some suggestions on the previous versions of the manuscript.


\begin{thebibliography}{10}

\bibitem{Tyson} Jayadev S.Athreya, Bruce Reznick, and Jeremy T.Tyson. %
\newblock Cantor set arithmetic.
\newblock {\em  American
Mathematical Monthly}, 126(1):4--17, 2019.

\bibitem{Yoccoz}
Carlos Gustavo~T. de~A.~Moreira and Jean-Christophe Yoccoz.
\newblock Stable intersections of regular {C}antor sets with large {H}ausdorff
  dimensions.
\newblock {\em Ann. of Math. (2)}, 154(1):45--96, 2001.


\bibitem{BABA1}
Bal\'{a}zs B\'{a}r\'{a}ny.
\newblock On some non-linear projections of self-similar sets in {$\Bbb{R}^3$}.
\newblock {\em Fund. Math.}, 237(1):83--100, 2017.

\bibitem{RSK}
 Rafael Alcaraz Barrera, Simon Baker and Derong Kong,
\newblock Entropy, topological transitivity, and dimensional properties of unique $q$-expansions. 
\newblock {\em Tran. AMS}, 370(2019), 3209--3258.



\bibitem{KarmaMartijn}
Karma Dajani and Martijn de~Vries.
\newblock Invariant densities for random {$\beta$}-expansions.
\newblock {\em J. Eur. Math. Soc. (JEMS)}, 9(1):157--176, 2007.


\bibitem{DKKL}
Karma Dajani, Vilmos Komornik, Derong Kong, and Wenxia Li.
\newblock Algebraic sums and products of univoque bases.
\newblock {\em Indag. Math. (N.S.)}, 29(4):1087--1104, 2018.

\bibitem{DJKL}
Karma Dajani, Kan Jiang, Derong Kong, and Wenxia Li.
\newblock Multiple expansions of real numbers with digits set $\{0,1,q\}$.
\newblock {\em Math.~Z.}, 291(3-4):1605--1619, 2019.

\bibitem{MK}
Martijn de~Vries and Vilmos Komornik.
\newblock Unique expansions of real numbers.
\newblock {\em Adv. Math.}, 221(2):390--427, 2009.

\bibitem{Farkas}
Abel Farkas.
\newblock Dimension approximation of attractors of graph directed {IFS}s by
  self-similar sets.
\newblock {\em Math. Proc. Cambridge Philos. Soc.}, 167(1):193--207, 2019.



\bibitem{GS}
Paul Glendinning and Nikita Sidorov.
\newblock Unique representations of real numbers in non-integer bases.
\newblock {\em Math. Res. Lett.}, 8(4):535--543, 2001.






\bibitem{Hutchinson}
John~E. Hutchinson.
\newblock Fractals and self-similarity.
\newblock {\em Indiana Univ. Math. J.}, 30(5):713--747, 1981.



\bibitem{MO}
Pedro Mendes and Fernando Oliveira.
\newblock On the topological structure of the arithmetic sum of two {C}antor
  sets.
\newblock {\em Nonlinearity}, 7(2):329--343, 1994.
\bibitem{ST}
K\'{a}roly Simon and Krystal Taylor.
\newblock Interior of sums of planar sets and curves \newblock {\em arXiv:1707.01420}, 2017.


\bibitem{HS} Hugo Steinhuas. \newblock Mowa W{\l }asno\'{s}\'{c} Mnogo\'{s}%
ci Cantora.
\newblock {\em Wector, 1-3. English translation in: STENIHAUS,
H.D.} 1985.

\bibitem{Palis}
Jacob Palis and Floris Takens.
\newblock {\em Hyperbolicity and sensitive chaotic dynamics at homoclinic
  bifurcations}, volume~35 of {\em Cambridge Studies in Advanced Mathematics}.
\newblock Cambridge University Press, Cambridge, 1993.
\newblock Fractal dimensions and infinitely many attractors.

\bibitem{PS}
Yuval Peres and Pablo Shmerkin.
\newblock Resonance between {C}antor sets.
\newblock {\em Ergodic Theory Dynam. Systems}, 29(1):201--221, 2009.









\bibitem{Sidorov}
Nikita Sidorov.
\newblock Almost every number has a continuum of {$\beta$}-expansions.
\newblock {\em Amer. Math. Monthly}, 110(9):838--842, 2003.


\bibitem{SN}
Nikita Sidorov.
\newblock Expansions in non-integer bases: lower, middle and top orders.
\newblock {\em J. Number Theory}, 129(4):741--754, 2009.


\end{thebibliography}
\end{document}